\documentclass[12pt]{amsart}
\usepackage[T1]{fontenc}
\usepackage[latin1]{inputenc}
\usepackage{typearea}
\usepackage{geometry}
\usepackage{ulem}
\usepackage{amsmath}
\usepackage{amssymb}
\usepackage{latexsym}
\usepackage{enumerate}
\usepackage{amsthm}
\usepackage[all]{xy}
\usepackage{hhline}
\usepackage{epsf} 
\usepackage{cite}
\usepackage{verbatim}
\usepackage{mathtools}
\usepackage{comment}
\usepackage{dsfont}
\usepackage{mathrsfs}

\usepackage{color}

\newtheorem{theorem}{{\sc Theorem}}[section]
\newtheorem{lemma}[theorem]{{\sc Lemma}}

\theoremstyle{remark}
\newtheorem{remark}[theorem]{{\sc Remark}}

\theoremstyle{definition}

\newcommand{\R}{\mathbb{R} }
\newcommand{\N}{\mathbb{N} }

\newcommand{\B}{\mathcal{B}}
\newcommand{\F}{\mathcal{F}}

\newcommand{\calN}{\mathcal{N}}

\newcommand{\Prob}{\mathbb{P}}
\newcommand{\E}{\mathbb{E}}
\newcommand{\la}{\lambda}

\newcommand{\Om}{\Omega}

\providecommand{\babs}[1]{\bigl\lvert #1\bigr\rvert}

\providecommand{\norm}[1]{\lVert #1\rVert}

\DeclareMathOperator{\Var}{Var}
\DeclareMathOperator{\supp}{supp}

\DeclareMathOperator{\HS}{H.S.}

\DeclareMathOperator{\TV}{TV}

\DeclareMathOperator{\esssup}{ess\,sup}
\DeclareMathOperator{\essinf}{ess\,inf}

\renewcommand{\phi}{\varphi}
\renewcommand{\epsilon}{\varepsilon}

\renewcommand{\rho}{\varrho}
\renewcommand{\P}{\Prob}
\renewcommand{\1}{\mathds{1}}

\begin{document}
\title[Existence and uniqueness of Stein kernels]{On existence and uniqueness of univariate Stein kernels}
\author{Christian D\"obler}
\thanks{\noindent Mathematisches Institut der Heinrich Heine Universit\"{a}t D\"usseldorf\\
Email: christian.doebler@hhu.de }
\begin{abstract}  
We completely characterize the class of univariate distributions allowing for a Stein kernel and illustrate our result by means of some concrete distributions. 
Moreover, we apply our findings to prove a quantitative version of the central limit theorem with optimal rate $n^{-1/2}$ in total variation distance for i.i.d. random variables whose distribution belongs to that class.
\end{abstract}

\maketitle

\section{Introduction and results}\label{intro}
Stein's method, originating from the works \cite{St72,St86}, is a collection of tools and techniques for estimating the distance between two probability distributions and for proving quantitative limit theorems, which has seen a vast amount of theoretical developments and many important applications in recent years. One particular and central object of study in this context is the so-called \textit{Stein kernel} associated to a probability distribution $\mu$ on $\R$, which has already been studied in Lecture VI of Stein's classical monograph \cite{St86}. 

Suppose that $X$ is a real-valued random variable on a probability space $(\Om,\F,\P)$ and with distribution $\mu:=\P_X$ on $(\R,\B(\R))$, where $\B(\R)$ denotes the $\sigma$-field of Borel subsets of $\R$. We further assume that $\E|X|<\infty$ and denote by $m:=\E[X]\in\R$ the mean of $X$. By definition, a Borel-measurable function $\tau_X:=\tau_\mu:\R\rightarrow\R$ is called a \textit{Stein kernel} for $X$ (or rather for $\mu$), if
\begin{equation}\label{defsk}
\E\bigl[\tau_X(X)f'(X)\bigr]=\E\bigl[(X-m)f(X)\bigr]  
 \end{equation}
 holds for all, say, continuously differentiable functions $f$ on $\R$ with a bounded derivative $f'$. Taking $f(x)=x$ one observes that \eqref{defsk} necessitates that $\E[X^2]<\infty$ and it further immediately implies that $\E[\tau_X(X)]=\Var(X)=:\sigma^2$, whenever $\tau_X$ exists.

\subsection{Motivation for studying the Stein kernel} 
The existence of a Stein kernel for $X$ may directly be applied in order to estimate the total variation distance of $\mu$ to the normal distribution $\calN(m,\sigma^2)$ with mean $m$ and variance $\sigma^2$. Indeed, it is known \cite{LRS,LNP} that in this case one has the \textit{Stein discrepancy bound}
\begin{align}\label{dtvbound}
d_{\TV}\bigl(\mu,\calN(m,\sigma^2)\bigr)&:=\sup_{B\in \B(\R)}\babs{\mu(B)-\calN(m,\sigma^2)(B)}\leq2\E\babs{\tau_X(X)-\sigma^2}\notag\\
&\leq2\sqrt{\Var(\tau_X(X))}.
\end{align}   
Note that, in view of Stein's lemma \cite{St72}, a Stein kernel for $\calN(m,\sigma^2)$ is given by the constant $\sigma^2$, showing that \eqref{dtvbound} may give an accurate bound, whenever $\tau_X$ is close to the constant $\sigma^2$ in $L^1(\mu)$. 

It has further been shown in \cite{NP09} that whenever $X$ is a smooth, centered functional of an isonormal Gaussian process $\mathbb{X}$ over the real separable Hilbert space $\mathfrak{h}$, then a Stein kernel for $X$ exists and is given by 
\begin{equation}\label{skmall}
 \tau_X(x)=\E\bigl[\langle DX,-DL^{-1}X\rangle_\mathfrak{h}\,\bigl|\, X=x\bigr],\quad x\in\R.
\end{equation}
Here, $D$ is the Malliavin derivative and $L^{-1}$ is the pseudo-inverse of the Ornstein-Uhlenbeck generator associated to $\mathbb{X}$. Together with the inequality \eqref{dtvbound} this can be considered the starting point of the so-called \textit{Malliavin-Stein method} on Gaussian spaces (see \cite{NP09,NPbook} for details). A particular, finite-dimensional version of the formula \eqref{skmall} had previously been employed in the article \cite{Cha}.

In recent works \cite{LNP, Saumard, ERS, GS} the Stein kernel has further successfully been exploited in the context of certain functional inequalities, covariance identities and estimates and approximations of (weighted) Poincar\'{e} constants.   \smallskip\\

As has been shown in previous works by the author \cite{DoeBeta,P4,D0}, the Stein kernel further naturally arises in the context of the \textit{exchangeable pairs approach} as follows: Suppose that $(W,W')$ is an exchangeable pair of real-valued random variables on $(\Om,\F,\P)$, that is $(W,W')$ and $(W',W)$ have the same distribution, and that one would like to approximate the distribution $\P_W$ of $W$ by a ``simpler'' distribution $\mu$ of another random variable $Z$, which might still have to be found. To this end, suppose further that there exist a non-increasing, real-valued function $\gamma$, a positive function $\eta$, both defined on a real interval $J=(a,b)$ such that the support of $\P_W$ is contained in the closure $\bar{J}$ of $J$, as well as ``small'', random remainder terms $R$ and $S$ such that 
\begin{align*}
\E\bigl[W'-W\,\bigl|\, W\bigr]&=\lambda\gamma(W)+R,\\
\frac{1}{2\lambda}\E\bigl[(W'-W)^2\,\bigl|\, W\bigr]&=\eta(W)+S.
\end{align*}   
Here, $\lambda$ is a positive and typically small real number. Then, in Section 2 of \cite{DoeBeta} it is suggested to approximate $\P_W$ by the distribution $\mu$ on $\B(\R)$ with probability density function
\begin{equation}\label{epdens}
p(x)=\frac{C}{\eta(x)}\exp\biggl(\int_{x_0}^x\frac{\gamma(t)}{\eta(t)}dt\biggr), \quad x\in J,
\end{equation} 
where $C$ is the normalizing constant and $x_0$ is a zero of $\gamma$. In many examples, the exchangeable pair $(W,W')$ at hand in fact satisfies the (approximate) \textit{linear regression property} 
\begin{equation*}
\E\bigl[W'-W\,\bigl|\, W\bigr]=\lambda(m-W)+R,
\end{equation*}
where $m=\int_a^b tp(t)dt$ is the mean of the distribution $\mu$, that is, the function $\gamma$ is actually given by $\gamma(x)=m-x$. In this case, as is shown in \cite{DoeBeta}, the function $\eta$ equals the Stein kernel $\tau_\mu$ for $\mu$ and a \textit{Stein operator} for the distribution $\mu$ is given by the first order linear differential operator $L$ with
\[(Lg)(x)=\tau_\mu(x)g'(x)+(m-x)g(x).\]  
Hence, the Stein kernel appears as a coefficient in the \textit{Stein equation} corresponding to $\mu$ and might, thus, also be referred to as a \textit{Stein coefficient} in this context. In the special case where $X$ is a centered and smooth functional of an isonormal Gaussian process and for $\gamma(x)=x$, the formula \eqref{epdens} for the density of $X$ had already been derived in the article \cite{NV}.\medskip\\

\subsection{Examples and existence of Stein kernels}\label{setup}

After motivating the interest in the Stein kernel $\tau_X$, we now come to the mathematically natural question about its existence and uniqueness. Throughout, denote by 
\[a:=\essinf(X)=\sup\bigl\{t\in\R\,:\,\P(X<t)=0\bigr\}\in[-\infty,+\infty)\]
 and by 
\[b:=\esssup(X)=\inf\bigl\{t\in\R\,:\,\P(X>t)=0\bigr\}\in(-\infty,+\infty]\]
the \textit{essential infimum} and the \textit{essential supremum} of $X$, respectively. Moreover, define the real interval $I:=(a,b)$ as well as the topological closure $\bar{I}=\overline{(a,b)}\subseteq \R$ of $I$ in $\R$. Note that $I=\emptyset$ if $X$ is $\P$-almost surely constant.

 For our existence and uniqueness results we will rely on the well-known \textit{Lebesgue decomposition}
\[\mu=\mu_{ac}+\mu_{d}+\mu_{sc}=\mu_{ac}+\mu_s\]
of $\mu$. Here, $\mu_{ac}$ is the \textit{absolutely continuous part} of $\mu$ which has a density with respect to the Lebesgue measure $\la$ on $\B(\R)$, $\mu_d$ is its \textit{discrete (or purely atomic) part}, that is, there exist a countable, possibly empty set $D\subseteq\R$ and positive real numbers $p_x$, $x\in D$, such that
\[\mu_d=\sum_{x\in D} p_x\delta_x.\]
In particular, $\sum_{x\in D}p_x=\mu_d(\R)\leq1$. Finally, $\mu_{sc}$ is the \textit{singular continuous part} of $\mu$, that is, the function $\R\ni x\mapsto\mu_{sc}((-\infty,x])\in[0,1]$ is continuous but $\mu_{sc}$ and $\lambda$ are mutually singular, i.e. there is a set $A\in\B(\R)$ such that $\mu_{sc}(A)=0=\lambda(\R\setminus A)$. For convenience, we further denote $\mu_s:=\mu_d+\mu_{sc}$ the part of $\mu$ that is singular with respect to $\lambda$. Of course, $\mu_{ac}\leq\mu$ is always absolutely continuous with respect to $\mu$ and, hence, by the Radon-Nikodym theorem, there exists a Borel-measurable and $\mu$-integrable function $h:\R\rightarrow[0,1]$ such that $\mu_{ac}=h\mu$, that is
\begin{equation}\label{rnmu}
\mu_{ac}(B)=\int_B h(t)d\mu(t),\quad B\in\B(\R).
\end{equation}

For a certain class of absolutely continuous  distributions $\mu$, that is with $\mu=\mu_{ac}$, the existence of a Stein kernel has essentially been known since the appearance of \cite{St86} (see e.g. \cite{LRS, DoeBeta} for a thorough review of the relevant literature). To be precise, if $X$ possesses a probability density function $p$ that is strictly positive on the interval $I=(a,b)=(\essinf(X),\esssup(X))$, then the function
\begin{equation}\label{skac}
\tau_X(x)=\frac{1}{p(x)}\int_a^x (m-t)p(t)dt=-\frac{1}{p(x)}\int_x^b (m-t)p(t)dt,\quad x\in\R,
\end{equation}  
is a Stein kernel for $X$. Moreover, the Stein kernel for $X$ is $\mu$-a.e. uniquely determined in this case. 

It turns out, however, that absolute continuity with respect to $\la$ is not necessary for a Stein kernel for $\mu$ to exist. Indeed, if $\mu=\delta_c$ is the Dirac measure at some $c\in\R$, i.e. if $X\equiv c$ $\P$-a.s., then it is easy to see that the function $\tau_X:\R\rightarrow\R$ with $\tau_X(c)=0$ and $\tau_X(t)=\gamma(t)$, $t\in\R\setminus\{c\}$, for any Borel-measurable function $\gamma:\R\rightarrow\R$, defines a Stein kernel for $\mu$. Moreover, it is clear that letting $\tau_X(c)=0$ is necessary for 
$\tau_X$ to be a Stein kernel for $X$ . Note further that, generally, \eqref{defsk} may only impose conditions on the function $\tau_X$ on the topological support $\supp(\mu)$ of the measure $\mu$ (see Remark \ref{mtrem} (a) below for a reminder of the definition of $\supp(\mu)$). 

For a non-degenerate example in this direction let 
\[\mu:=\frac14\delta_1+\frac14\delta_{-1}+\frac12 \mathcal{U}([-1,1]),\]
where, for real numbers $c<d$, $\mathcal{U}([c,d])$ denotes the uniform distribution on the interval $[c,d]$. Then, it may be checked by elementary computations that a Stein kernel $\tau_\mu$ for $\mu$ is given by 
\[\tau_{\mu}(t)=\Bigl(1+\frac12\bigl(1-t^2\bigr) \Bigr)\1_{(-1,1)}(t).\]

However, not every distribution on $\R$ with a finite second moment allows for a Stein kernel. If $X$ is a symmetric Rademacher variable, for instance, that is, $\P(X=1)=\P(X=-1)=1/2$, then it is easy to see that there is no Stein kernel for $X$. Indeed, for any measurable function $\tau$ on $\R$ and every continuously differentiable function $f$ one has 
\[\E\bigl[Xf(X)\bigr]=\frac12\bigl(f(1)-f(-1)\bigr)\]
and
\[\E\bigl[\tau(X)f'(X)\bigr]=\frac12\bigl(\tau(1)f'(1)-\tau(-1)f'(-1)\bigr).\]
Now, if $\tau$ were a Stein kernel for $X$, then choosing, successively, $f(x)=x$ and $f(x)=x^3$ would yield $\tau(1)-\tau(-1)=2$ and $\tau(1)-\tau(-1)=\frac23$, which cannot both hold at the same time. 
 In such a simple example, the non-existence of a Stein kernel is, of course, obvious but from our main result, Theorem \ref{maintheo} below, it will also follow that the Cantor distribution on $[0,1]$, like in fact any other non-degenerate singular distribution, has no Stein kernel either, which might be less obvious.

Moreover, not even every absolutely continuous distribution $\mu$ allows for a Stein kernel. If, for instance, 
\[\mu=\frac12\mathcal{U}([-2,-1])+\frac12\mathcal{U}([1,2]),\]
then $X\sim\mu$ satisfies $\E[X]=0$ and $\Var(X)=7/3$. Moreover, a partial integration yields that 
\[\E[Xf(X)]=-\E[X^2f'(X)]+\frac12\bigl(f(-1)-f(1)\bigr)+2\bigl(f(2)-f(-2)\bigr)\]
for any continuously differentiable function $f$ on $\R$. Hence, if $\tau_X$ were a Stein kernel for $X$, then one would have the identity
\[\E\Bigl[\bigl(X^2+\tau_X(X)\bigr)f'(X)\Bigr]=\frac12\bigl(f(-1)-f(1)\bigr)+2\bigl(f(2)-f(-2)\bigr)\]
for any continuously differentiable $f$. By a functional analytic argument this would imply that $\tau_X(t)=-t^2$ for $\lambda$-a.e. $t\in [-2,1]\cup[1,2]$. But then, by taking $f(x)=x$ it would follow that 
\[-\frac{7}{3}=\E[-X^2]=\E\bigl[\tau_X(X)\bigr]=\E[X^2]=\frac73,\]
which is impossible. Hence, a Stein kernel for $\mu$ cannot exist.

In view of the well-known formula \eqref{skac} and the positive and negative examples above, it is quite surprising that the general question of existence and uniqueness of univariate Stein kernels has not been addressed in the literature so far. 
In fact, most of the above cited references content themselves with distributions for which one of the representations \eqref{skmall} or \eqref{skac} is valid. It is the main purpose of this note to fill this gap in Stein's method theory.\smallskip\\

Before stating our main result, we would like to mention the important fact that the concept of a Stein kernel can by generalized to higher dimensions. Thus, if $\nu$ is a probability distribution on $(\R^d,\B(\R^d))$ such that $\int_{\R^d}\norm{x}^2d\nu(x)<\infty$, then a Stein kernel for $\nu$ is a matrix-valued, measurable function $\tau_\nu:\R^d\rightarrow\R^{d\times d}$ such that 
\begin{equation*}
\int_{\R^d}\langle x-a,f(x)\rangle d\nu(x)=\int_{\R^d}\langle \tau_\nu(x),Df(x)\rangle_{\HS} d\nu(x)
\end{equation*} 
holds for all smooth mappings $f:\R^d\rightarrow\R^d$ such that $\int_{\R^d}(\norm{f}^2+\norm{Df}^2_{\HS})d\nu<\infty$. Here, $a=\int_{\R^d}x d\nu(x)\in\R^d$ is the mean vector of $\nu$, $\HS$ refers to the Hilbert-Schmidt inner product and norm, respectively, and $\langle \cdot,\cdot\rangle$ denotes the standard inner product on $\R^d$. A slightly weaker definition assumes this identity only for gradient fields $f=\nabla g$, where $g$ is a twice continuously differentiable function $g$ on $\R^d$. In \cite{CFP} it has been shown that a probability distribution on $(\R^d,\B(\R^d))$ has a Stein kernel, whenever it is absolutely continuous with respect to the $d$-dimensional Lebesgue measure $\lambda^d$ on $\B(\R^d)$ 
 and satisfies a (converse weighted) Poincar\'{e} inequality. It should be noted that, contrary to the univariate case, for $d\geq2$, the existence of a Stein kernel is non-trivial even for absolutely continuous distributions $\nu$ with a smooth, log-concave density that is strictly positive on the support of $\nu$. This is mostly due to the fact that such a simple formula as \eqref{skac} does not seem to exist for $d\geq2$. In this context, it is worthwhile mentioning that the existence proof from \cite{CFP} invokes the Lax-Milgram theorem from functional analysis and is, hence, not constructive. Some important progress on the form of a Stein kernel for multivariate absolutely continuous distributions has been made in the recent interesting paper \cite{Fathi}, though. From Theorem \ref{maintheo} below, it might be possible to make an educated guess about a general existence criterion for a Stein kernel $\tau_\nu$ in arbitrary dimension and also about its general form. We prefer to leave this as an interesting problem for future work, however. In this regard it should further be mentioned that in (the arXiv version of) the article \cite{NPS1} it has been suggested that, for $d\geq2$, a Stein kernel for $\nu$ might no longer be a.e. unique.    

 \subsection{Results}
 
 We now state the main result of this note, which completely characterizes all univariate distributions possessing a Stein kernel.

\begin{theorem}\label{maintheo}
Suppose that $X$ is a real-valued random variable, defined on a probability space $(\Om,\F,\P)$, with non-degenerate distribution $\mu=\P_W$ and such that $\E[X^2]<\infty$. Then, there exists a Stein kernel $\tau_X$ for $X$, if and only if the absolutely continuous part $\mu_{ac}$ of $\mu$ has a Lebesgue density $p=\frac{d\mu_{ac}}{d\lambda}$ which is strictly positive on $I$. In this case, with $m=\E[X]$ and $0<\sigma^2=\Var(X)$, the Stein kernel $\tau_X$ is $\mu$-a.e. unique and a version of it is given by  
\begin{equation}\label{formtau}
\tau_X(t)=\sigma^2\frac{q(t)}{p(t)}h(t),\quad t\in I,
\end{equation}
and $\tau_X(t)=0$ for $t\in\R\setminus I$. Here, $q$	is the probability density function of the $X$-non-zero biased distribution defined in \eqref{densnz} below and $h=\frac{d\mu_{ac}}{d\mu}$ is a Radon-Nikodym derivative of $\mu_{ac}$ with respect to $\mu$ as in \eqref{rnmu}. In particular, $\tau_X$ is $\lambda$-a.e. positive on $\bar{I}$. Moreover, it necessarily holds that $\tau_X(t)=0$ for $\mu_s$-a.e $t\in\R$. In particular, one has that $\tau_X(c)=0$, whenever $\P(X=c)>0$.
\end{theorem}

\begin{remark}\label{mtrem}
\begin{enumerate}[(a)]
\item Recall that the \textit{(topological) support} $\supp(\nu)$ of a (nonnegative) measure $\nu$ on $\B(\R)$ may be defined as 
 \[\supp(\nu):=\bigcap_{C\in\mathcal{C_\nu}} C,\]
 where $\mathcal{C_\nu}$ is the collection of all closed subsets of $\R$ such that $\nu(\R\setminus C)=0$. In particular, $\supp(\nu)$ is a closed subset of $\R$ and by second-countability it further follows that $\nu(\R\setminus\supp(\nu))=0$.

Now note that the condition that $p=\frac{d\mu_{ac}}{d\lambda}$ is $\lambda$- a.e. strictly positive on $I$ in particular implies that the (topological) support of $\mu_{ac}$ (and, hence, that of $\mu$ as well) is equal to the closed interval $\bar{I}$. 
\item The converse statement, however, in general is not correct so that we cannot replace the condition that $p>0$ $\lambda$-a.e. in Theorem \ref{maintheo} with the condition that the support of $\mu_{ac}$ is equal to $\bar{I}$. For a counterexample, take an enumeration $(r_n)_{n\in\N}$ of the rational numbers and define the intervals $J_n:=[r_n-2^{-n},r_n+2^{-n}]$, $n\in\N$. Furthermore, let $p:=c \sum_{n=1}^\infty \1_{J_n}$, where $c>0$ is the normalizing constant. Then, $P:=\{x\in\R\,:\, p(x)>0\}=\bigcup_{n\in\N} J_n$, the topological support of $\mu=p\lambda$ is given by $\bar{P}=\R$, since the rational numbers are dense in $\R$, but $p$ is not $\lambda$-a.e. positive, because $\lambda(P)\leq\sum_{n=1}^\infty 2^{-n+1}=2$.
\item In \cite[Theorem 3.1]{NV} it is proved for a smooth and centered functional $X$ of an isonormal Gaussian process that the law $\mu$ of $X$ is absolutely continuous with respect to $\lambda$, if and only if $\tau_X(X)$ is $\P$-a.s. strictly positive, where $\tau_X$ is given in \eqref{skmall}. From Theorem \ref{maintheo} we can infer that this result cotinues to hold in general: If $X$ has a Stein kernel $\tau_X$, then $\mu$ is absolutely continuous, if and only if $\tau_X(X)$ is $\P$-a.s. strictly positive. Indeed, if there is $A\in\B(\R)$ such that $\mu_s(A)>0$, then $\P(\tau_X(X)=0)=\mu(\tau_X=0)\geq \mu(A)\geq\mu_s(A)>0$.
\end{enumerate}
\end{remark}

The next result applies the findings of Theorem \ref{maintheo} to derive a rate-optimal, quantitative version of the classical CLT in total variation distance. 

\begin{theorem}\label{app}
Let the probability distribution $\mu$ be such that for $X\sim\mu$ it holds that $\E[X^2]<\infty$, $\sigma^2:=\Var(X)>0$ and assume further that $\mu_{ac}$ has a Lebesgue density, which is strictly positive on $I=(\essinf(X),\esssup(X))$. Moreover, suppose that, for some $n\in\N$, $X_1,\dotsc,X_n$, are independent random variables on a probability space $(\Om,\F,\P)$ such that $X_j\sim\mu$ for $1\leq j\leq n$ and define $S_n:=\sum_{j=1}^n X_j$. Then, with $m:=\E[X_1]$, $S_n^*:=\sigma^{-1}n^{-1/2}(S_n-nm)$ and $Z$ a standard normal random variable it holds that
\[d_{\TV}(S_n^*,Z)\leq  \frac{2}{\sigma^2\sqrt{n}}\Bigl(\Var\bigl(\tau_X(X)\bigr)\Bigr)^{1/2},\]
where $\tau_X$ is the (by Theorem \ref{maintheo} existing) Stein kernel for $X$.
\end{theorem}

It was already proved in the paper \cite{SM} that a total variation bound in the classical CLT with rate of order $n^{-1/2}$ even holds under the minimal condition that $\mu_{ac}(\R)>0$. However, the constants in the bound from \cite{SM} are not explicit, whereas the variance of $\tau_X(X)$ might be easily computable in concrete situations. We refer to the article \cite{BCG} and the references therein for more details about known results in this direction.

\section{Proofs}\label{proofs}
 By the above observations, we may and will from now on focus on probability distributions $\mu$ for $X$ such that $\E[X^2]<\infty$ and $0<\sigma^2:=\Var(X)$. As before, we let $m:=\E[X]$. Furthermore, we continue to use the notation introduced in Subsection \ref{setup}.

An important tool for our proofs is the so-called \textit{$X$-non-zero biased distribution} associated to $X$ (or rather to $\mu$). By definition, this is another probability distribution $\mu^{nz}$ on $(\R,\B(\R))$ such that for $X^{nz}\sim\mu^{nz}$ one has
\begin{align}\label{xnzdef}
 \sigma^2\E\bigl[f'(X^{nz})\bigr]= \sigma^2 \int_\R f'(t)d\mu^{nz}(t)=\E\bigl[(X-m)f(X)\bigr]
\end{align}
for all continuously differentiable functions $f$ on $\R$ with bounded derivative $f'$. This distributional transformation is a straightforward extension of the well-known \textit{zero-bias transformation} by Goldstein and Reinert \cite{GolRei97} to not necessarily centered random variables, which has been introduced in the work \cite{Doe17} by the author. The random variable $X^{nz}$ may for instance be constructed as $X^{nz}:=m+Y$, where $Y$ has the $(X-m)$-zero biased distribution from \cite{GolRei97}.

Of the many interesting properties of this distributional transformation, we will only need the fact that $\mu^{nz}$ exists, is always absolutely continuous with respect to $\lambda$ and that
\begin{equation}\label{densnz}
   q(t):=\frac{1}{\sigma^2}\E\bigl[(X-m)\1_{\{X\geq t\}}\bigr],\quad t\in\R,
  \end{equation}
 is a probability density function of $\mu^{nz}$, which is strictly positive on $I$ and which vanishes on $\R\setminus\bar{I}$. In particular, the measures $\mu^{nz}$ and $\lambda$ are equivalent on the interval $\bar{I}$. These facts can easily be checked by means of direct computations, or else inferred from \cite[Lemma 2.1]{GolRei97} (see also \cite[Subsection 2.3.3]{CGS}). 
 \smallskip\\

We prepare the proof of Theorem \ref{maintheo} with three auxiliary statements.

\begin{lemma}\label{le1}
Let $\tau:\R\rightarrow\R$ be Borel-measurable. Then, $\tau$ is a Stein kernel for $X$, if and only if $\mu^{nz}=\sigma^{-2} \tau\mu$. In particular, a Stein kernel $\tau_X$ for $X$ exists, if and only if $\mu^{nz}$ is absolutely continuous with respect to $\mu$. In this case, $\tau_X$ is $\mu$-a.s. uniquely determined, $\mu$-a.s. nonnegative and $\lambda$-a.e. positive on $\bar{I}$.  
\end{lemma}

\begin{proof}
Suppose first that $\tau$ is a Stein kernel for $X$ and let $f:\R\rightarrow\R$ be continuously differentiable with bounded derivative $f'$. Then, by assumption and by Fubini's theorem,
\begin{align*}
 &\int_\R f'(t)\tau(t)d\mu(t)=\E\bigl[\tau(X)f'(X)\bigr]=\E\bigl[(X-m)f(X)\bigr]=\E\bigl[(X-m)(f(X)-f(0))\bigr]\notag\\
 &=\E\biggl[(X-m)\int_0^X f'(t)dt\biggr]=\int_\R f'(t)\E\bigl[(X-m)\bigl(\1_{\{0\leq t\leq X\}}-\1_{\{X<t<0\}}\bigr)\bigr]dt\notag\\
&=\int_\R f'(t)\E\bigl[(X-m)\1_{\{X\geq t\}}\bigr]dt=\int_\R f'(t)\sigma^2 q(t)dt= \sigma^2 \int_\R f'(t)d\mu^{nz}(t).
 \end{align*}
 Hence, for all $g\in C_b(\R)$, the class of all bounded and continuous functions on $\R$, we have 
\begin{equation*}
 \int_\R g(t)\tau(t)d\mu(t)=\sigma^2 \int_\R g(t)d\mu^{nz}(t).
\end{equation*}
Thus, the Riesz representation theorem yields that the signed measure $\nu=\sigma^{-2}\tau\mu$ and the probability measure $ \mu^{nz}$ coincide, implying both nonnegativity and uniqueness of $\tau_X$  $\mu$-almost surely. Moreover, letting 
\begin{equation}\label{defN}
 N:=\bigl\{t\in \bar{I}\,:\,\tau(t)\leq0\bigr\}
\end{equation}
we have 
\begin{equation}\label{eq2}
 0\leq\mu^{nz}(N)=\sigma^{-2} \int_N \tau(t)d\mu(t)\leq0
\end{equation}
and, thus, $\mu^{nz}(N)=0$. Since, by the positivity of $q$ on $I$, $\mu^{nz}$ and the restriction of $\lambda$ to $\bar{I}$ are equivalent, it follows that $\lambda(N)=0$ as desired. 

Conversely, suppose now that $\mu^{nz}=\sigma^{-2} \tau\mu$ and let again $f$ be continuously differentiable with bounded derivative $f'$. Then,
\begin{align*}
\E\bigl[\tau(X)f'(X)\bigr]=\int_\R f'(t)\tau(t)d\mu(t)=\sigma^{2}\int_\R f'(t)d\mu^{nz}(t)
=\E\bigl[(X-m)f(X)\bigr],
\end{align*}
proving that $\tau$ is indeed a Stein kernel for $X$.
\end{proof}

\begin{lemma}\label{le2}
Suppose that $\tau_X:\R\rightarrow\R$ is a Stein kernel for $X$. Then, $\tau_X=0$ $\mu_s$-a.e. In particular, $\tau_X(c)=0$, whenever $\P(X=c)>0$.
\end{lemma}

\begin{proof}
Let $A\in\B(\R)$ be such that $\lambda(A)=0=\mu_s(\R\setminus A)$. With the set $N$ as in \eqref{defN} we let $P:=\bar{I}\setminus N=\{t\in \bar{I}\,:\, \tau_X(t)>0\}$. 
Since we already know from Lemma \ref{le1} that $\tau_X$ is nonnegative $\mu_s$-a.e., it remains to show that $\mu_s(P)=0$. Using that $\mu_{ac}(A)=\mu^{nz}(A)=0$, by Lemma \ref{le1} we have
\begin{align*}
\mu_s(P)&=\mu_s(A\cap P)=\mu(A\cap P)=\int_{A\cap P}\tau_X(t)^{-1}\tau_X(t)d\mu(t)\\
&=\sigma^2\int_{A\cap P}\tau_X(t)^{-1} d\mu^{nz}(t)=0,
\end{align*}
as claimed.
\end{proof}

\begin{lemma}\label{le3}
Suppose that $\tau_X:\R\rightarrow\R$ is a Stein kernel for $X$. Then, a Lebesgue density $p$ for the absolutely continuous part $\mu_{ac}$ of $\mu$ is given by 
\begin{equation}\label{densform}
   p(t):=\begin{cases}
          \frac{\sigma^2 q(t)}{\tau_X(t)}\,,&t\in \bar{I}\text{ and }\tau_X(t)\not=0\\
          0\,,&t\notin \bar{I}\text{ or }\tau_X(t)=0,
         \end{cases}
  \end{equation}
	where the probability density function $q$ of $\mu^{nz}$ has been defined in \eqref{densnz}.
	In particular, it holds that $p(t)>0$ for $\lambda$-a.e. $t\in I$. 
	\end{lemma}

\begin{proof}
We use the notation introduced in the proofs of Lemmas \ref{le1} and \ref{le2}. Let $B\in\B(\R)$ be given. If $\lambda(B)=0$, then 
\[0=\int_B p(t)dt=\mu_{ac}(B).\]
If, on the other hand, $\lambda(B)>0$ then we have $\mu_s(B\setminus N)=\mu_s((B\cap\bar{I})\setminus N)=\mu_s(B\cap P)=0$ and, hence, as $\lambda(N)=0$ by Lemma \ref{le1},
\begin{align*}
 \int_B p(t)dt&
=\int_{(B\cap \bar{I})\setminus N} p(t)dt=\int_{(B\cap \bar{I})\setminus N}\frac{\sigma^2 q(t)}{\tau_X(t)}dt
=\sigma^2\int_{(B\cap \bar{I})\setminus N}\tau_X(t)^{-1}d\mu^{nz}(t)\notag\\
&=\int_{(B\cap \bar{I})\setminus N} \tau_X(t)^{-1}\tau_X(t)d\mu(t)=\int_\R\1_{(B\cap \bar{I})\setminus N} d\mu(t)=\mu(B\setminus N)\\
 &=\mu_{ac}(B\setminus N)+\mu_s(B\setminus N) =\mu_{ac}(B\setminus N)=\mu_{ac}(B).
\end{align*}
The final equality follows from 
\[\mu_{ac}(B)=\mu_{ac}(B\setminus N)+\mu_{ac}(B\cap N)=\mu_{ac}(B\setminus N),\]
again as $\lambda(N)=0$. Hence, we have proved that $\mu_{ac}=p\lambda$. Finally, since $\tau_X$ is $\lambda$-a.e. positive on $\bar{I}$ by Lemma \ref{le1} and the same holds for $q$ it thus follows that also $p$ is $\lambda$-a.e. positive on $\bar{I}$.
\end{proof}

\begin{proof}[Proof of Theorem \ref{maintheo}]
Suppose first that there exists a Stein kernel $\tau_X$ for $X$. Then, by Lemma \ref{le3}, the absolutely continuous part $\mu_{ac}$ of $\mu$ has a density $p$ which is positive on $I$. Moreover, the additional assertions on $\tau_X$ follow from Lemma \ref{le2}. 

Suppose now that, conversely, $\mu_{ac}$ has a Lebesgue density $p$ which is strictly positive on $I$ and define $\tau_X(t)$ by \eqref{formtau} for $t\in I$ and let $\tau_X(t)=0$ for $t\notin I$. Then, $\tau_X$ is Borel-measurable and nonnegative and the computation below with $f(x)=x$ shows that it is also $\mu$-integrable. Moreover, by the positivity of $p$ on $I$, the restriction $\lambda|I$ of $\lambda$ on $I$ is absolutely continuous with respect to $\mu_{ac}$ with density $\frac{d\lambda|I}{d\mu_{ac}}= p^{-1}|I$. Hence, for each continuously differentiable function $f$ on $\R$ with bounded derivative one has
\begin{align*}
&\E\bigl[\tau_X(X)f'(X)\bigr]=\int_\R\tau_X(t) f'(t) d\mu(t)=\int_I\tau_X(t) f'(t) d\mu(t)\\
&= \sigma^2\int_I \frac{q(t)}{p(t)}f'(t)h(t) d\mu(t)=\sigma^2\int_I \frac{q(t)}{p(t)}f'(t) d\mu_{ac}(t)=\sigma^2\int_I q(t)f'(t) d\lambda(t)\\
&=\sigma^2\int_\R q(t)f'(t) d\lambda(t)=\sigma^2 \int_\R f'(t) d\mu^{nz}(t)=\E\bigl[(X-m)f(X)\bigr].
\end{align*} 
Thus, $\tau_X$ is indeed a Stein kernel for $X$ as claimed.
\end{proof}

\begin{proof}[Proof of Theorem \ref{app}]
As on page 39 in \cite{LRS} we have that the function $\tau$ with 
\[\tau(x):=\frac{1}{n\sigma^2}\sum_{j=1}^n \E\bigl[\tau_X(X_j)\,\bigl|\, S_n^*=x\bigr],\quad x\in\R,\]
is a Stein kernel for $S_n^*$ and that 
\[\E\Bigl[\bigl(1-\tau(S_n^*)\bigr)^2\Bigr]\leq \frac{1}{n\sigma^4}\Var\bigl(\tau_X(X)\bigr).\]
Hence, the claim follows from the bound in \eqref{dtvbound}.
\end{proof}

\normalem
\bibliography{steinkernel}{}
\bibliographystyle{alpha}
\end{document}